\documentclass[a4paper, 10pt, reqno]{amsart}
\usepackage[latin1]{inputenc}
\usepackage[francais]{babel}
\usepackage{amsfonts}
\usepackage{amsmath}
\usepackage{amssymb}
\usepackage{amsthm}
\usepackage{amsxtra}
\usepackage{amstext}
\usepackage{eufrak}
\usepackage{latexsym}
\usepackage[T1]{fontenc}
\usepackage[all]{xy}

\def\FF{\mathbb{F}}
\def\FFF{\overline{\mathbb{F}}}

\def\NN{\mathbb{N}}

\def\QQ{\mathbb{Q}}

\def\ZZ{\mathbb{Z}}

\def\1{\mathbf{1}}

\newtheorem{theorem}{Th\'eor\`eme}[section]
\newtheorem{corollary}[theorem]{Corollaire}
\newtheorem{proposition}[theorem]{Proposition}
\newtheorem{lemma}[theorem]{Lemme}
\newtheorem{rem}[theorem]{Remarque}
\newenvironment{remark}{\begin{rem} \rm }{\end{rem}}

\numberwithin{equation}{section}

\newcommand{\defeq}{\stackrel{\textrm{\tiny{def}}}{=}}
\newcommand{\maq}[4]{\left(\begin{array}{cc} #1&#2\\#3&#4\end{array}\right)}
\newcommand{\fpbar}{\overline{\mathbb{F}}_{p}}
\newcommand{\sym}[1]{\mathrm{Sym}^{#1}\fpbar^2}
\newcommand{\gl}[1]{\mathrm{GL}_{2}(#1)}

\newcommand{\fp}{\mathbb{F}_{p}}

\newcommand{\N}{\mathbf{N}}
\newcommand{\qp}{\mathbb{Q}_{p}}

\newcommand{\zp}{\mathbb{Z}_{p}}

\newcommand{\ind}[3]{\mathrm{Ind}_{#1}^{#2}#3}

\newcommand{\kop}[1]{K_{0}(p^{#1})}

\newcommand{\stand}[2]{\lbrack #1,#2\rbrack}

\title{Structure interne des repr\'esentations modulo $p$ de $SL_{2}(\QQ_{p})$}
\author{Ramla Abdellatif \and Stefano Morra}
\date{}
\address{Laboratoire de Math\'ematiques - B\^atiment 425 - Universit\'e Paris-Sud XI, 91 405 Orsay; \mbox{Ramla.Abdellatif@math.u-psud.fr}}
\address{Institut de Math\'ematiques et de Mod\'elisation de Montpellier - place Eug\`ene Bataillon, case courrier 051, 34095 Montpellier ; \mbox{stefano.morra@math.univ-montp2.fr}}
\begin{document}
\maketitle
\begin{abstract}
Soit $p \geq 5$ un nombre premier. \`A l'aide de travaux ant\'erieurs des deux auteurs, nous d\'eterminons la filtration par le $SL_{2}(\ZZ_{p})$-socle des $\FFF_{p}$-repr\'esentations lisses irr\'eductibles de $SL_{2}(\QQ_{p})$.
\end{abstract}

\section{Introduction}
\label{Intro}
Soit $p$ un nombre premier, que l'on choisit comme uniformisante de $\ZZ_{p}$, soit $F$ une extension finie de $\QQ_{p}$ et soit $\FFF_{p}$ une cl\^oture alg\'ebrique du corps r\'esiduel $\FF_{p}$ de $\ZZ_{p}$. Les r\'ecents travaux de Breuil-Pa$\check{\text{s}}$k$\bar{\text{u}}$nas \cite{BP} et Hu \cite{yongquan} ayant trait \`a la recherche d'une classification des repr\'esentations lisses irr\'eductibles de $GL_{2}(F)$ sur $\FFF_{p}$ mettent en valeur l'importance cruciale de la compr\'ehension des extensions entre repr\'esentations lisses irr\'eductibles d'un sous-groupe ouvert compact maximal hyperspecial\footnote{De telles repr\'esentations sont parfois appel\'ees \emph{poids de Serre} dans la litt\'erature.} $K$ fix\'e de $GL_{2}(F)$ qui apparaissent comme sous-quotients d'une $\FFF_{p}$-repr\'esentation lisse irr\'eductible donn\'ee.

Dans cette optique, nous nous proposons d'\'etudier la filtration par le $SL_{2}(\ZZ_{p})$-socle des repr\'esentations lisses irr\'eductibles de $SL_{2}(\QQ_{p})$ sur $\FFF_{p}$. Pour ce faire, nous allons utiliser certains travaux ant\'erieurs des deux auteurs : en effet, le premier auteur donne dans \cite{ramla} une classification compl\`ete des $\FFF_{p}$-repr\'esentations lisses irr\'eductibles de $SL_{2}(\QQ_{p})$ tandis que, pour $p$ est impair, le second auteur d\'etermine dans \cite{stefano} la filtration par le $GL_{2}(\ZZ_{p})$-socle des $\FFF_{p}$-repr\'esentations lisses irr\'eductibles de $GL_{2}(\QQ_{p})$.

Avant d'\'enoncer notre r\'esultat, nous rappelons que pour tout entier $n \in \ZZ$, on note $\chi^{s}_{n}$ le $\FFF_{p}$-caract\`ere du groupe $B(\FF_{p})$ des matrices triangulaires sup\'erieures de $GL_{2}(\FF_{p})$ d\'efini par 
\begin{equation}
\label{defchins}
\chi^{s}_{n}\left(\left(\begin{array}{cc} a & b \\ 0 & d\end{array}\right)\right) := d^{n}\ , 
\end{equation}
que l'on d\'esigne par $\frak{a}$ le $\FFF_{p}$-caract\`ere de $B(\FF_{p})$ d\'efini par $\frak{a}\left(\left(\begin{array}{cc} a & b \\ 0 & d\end{array}\right)\right) := ad^{-1}$, par $\omega$ le $\FFF_{p}$-caract\`ere lisse de $\QQ^{\times}_{p}$ trivial en $p$ dont la restriction \`a $\ZZ_{p}^{\times}$ est donn\'ee par l'application de r\'eduction modulo $p$, et que pour tout coefficient $\lambda \in \FFF_{p}^{\times}$ non nul, on note $\mu_{\lambda}$ le $\FFF_{p}$-caract\`ere lisse non ramifi\'e de $\QQ^{\times}_{p}$ envoyant $p$ sur $\lambda$.
\begin{theorem}
\label{thmppaldelanote}
Soit $r \in \{0, \ldots , p-1\}$ et soit $\lambda \in \FFF^{\times}_{p}$. Notons $I_{S}$ le sous-groupe d'Iwahori standard de $SL_{2}(\QQ_{p})$, $K_{S} := SL_{2}(\ZZ_{p})$ et $B_{S}$ le sous-groupe des matrices triangulaires sup\'erieures de $SL_{2}(\QQ_{p})$.
\begin{enumerate}
\item La filtration par le $K_{S}$-socle de l'induite parabolique $\ind{B_{S}}{G_{S}}(\mu_{\lambda}\omega^{p-1-r})$ est donn\'ee par 
	\begin{eqnarray*}
	\mathrm{SocFil}(\ind{I_S}{K_S}(\chi_{r}^{s}))\textbf{---}\mathrm{SocFil}(\ind{I_S}{K_S}(\chi_{r-2}^s))\textbf{---}
	\mathrm{SocFil}(\ind{I_S}{K_S}(\chi_{r-4}^s))\textbf{---}\dots \ .
	\end{eqnarray*}
\item La filtration par le $K_{S}$-socle de la repr\'esentation de Steinberg $St_{S}$ est donn\'ee par  
	\begin{eqnarray*}
	\sym{p-1}\textbf{---}\mathrm{SocFil}(\ind{I_{S}}{K_{S}}(\frak{a}))\textbf{---}\mathrm{SocFil}(\ind{I_{S}}{K_{S}}(\frak{a}^{2}))\textbf{---} 
	\mathrm{SocFil}(\ind{I}{K}(\frak{a}^{3}))\textbf{---} \dots \ .
	\end{eqnarray*}
\item La filtration par le $K_{S}$-socle de la repr\'esentation supersinguli\`ere $\pi_{r}$ est donn\'ee par 
	\begin{eqnarray*}
	\sym{r}\textbf{---}\mathrm{SocFil}(\ind{I_S}{K_S}(\chi_{-r-2}^s))\textbf{---}
	\mathrm{SocFil}(\ind{I_S}{K_S}(\chi_{-r-4}^s))\textbf{---}\dots \ . 
	\end{eqnarray*}
\end{enumerate}
\end{theorem}
\noindent  Cet article est organis\'e de la fa\c{c}on suivante : on commence par rappeler les r\'esultats de \cite{ramla} et \cite{stefano} ayant inspir\'e l'\'etude men\'ee ici, et qui traitent respectivement des repr\'esentations modulo $p$ de $SL_{2}(\QQ_{p})$ et de la structure interne des repr\'esentations modulo $p$ de $GL_{2}(\QQ_{p})$. On d\'emontre ensuite deux r\'esultats techniques concernant le comportement des $GL_{2}(\ZZ_{p})$-extensions apr\`es restriction \`a $SL_{2}(\ZZ_{p})$, dont nous d\'eduisons finalement une preuve simple du Th\'eor\`eme \ref{thmppaldelanote}.

\section{Pr\'eliminaires}
\label{RappelsdeNous}
\subsection{Notations}
\label{Preliminairesnotations}
On fixe un nombre premier $p$. On note $G := GL_{2}(\QQ_{p})$ de sous-groupe ouvert compact maximal hypersp\'ecial $K := GL_{2}(\ZZ_{p})$ et de centre $Z \simeq \QQ^{\times}_{p}$ dont le pro-$p$-sous-groupe maximal est not\'e $Z_{1}$. On d\'esigne par $B$ le sous-groupe de Borel form\'e des matrices triangulaires sup\'erieures de $G$ et par $U$ son radical unipotent. L'application de r\'eduction modulo $p$ induit naturellement une surjection de $K$ sur $GL_{2}(\FF_{p})$ qui nous permet de d\'efinir le sous-groupe d'Iwahori standard $I$ (resp. : son pro-$p$-radical $I(1)$) comme l'image r\'eciproque par cette surjection du sous-groupe des matrices triangulaires sup\'erieures (resp. : et unipotentes) de $GL_{2}(\FF_{p})$. Plus g\'en\'eralement, pour tout entier $n \geq 1$, on d\'efinit $K_{0}(p^{n})$ comme le sous-groupe des \'el\'ements de $K$ s'\'ecrivant sous la forme $\left(\begin{array}{cc} a & b \\ p^{n}c & d\end{array}\right)$ avec $a, b, c, d \in \ZZ_{p}$.\\

\noindent De m\^eme, on note $G_{S} := SL_{2}(\QQ_{p})$ de sous-groupe ouvert compact maximal hypersp\'ecial standard $K_{S} = SL_{2}(\ZZ_{p})$, de sous-groupe de Borel $B_{S} := G_{S} \cap B$, de sous-groupe d'Iwahori standard $I_{S} = G_{S} \cap I$ et de pro-$p$-Iwahori standard $I_{S}(1) = G_{S} \cap I(1)$.\\

\noindent Comme nous l'avons rappel\'e dans l'introduction, on d\'esigne par $\omega$ le $\FFF_{p}$-caract\`ere lisse de $\QQ_{p}^{\times}$ qui est trivial sur $p$ et dont l'action sur $\ZZ_{p}^{\times}$ est d\'efinie par l'application de r\'eduction modulo $p$. Pour tout coefficient $\lambda \in \FFF_{p}$ non nul, on note $\mu_{\lambda} : \QQ^{\times}_{p} \to \FFF^{\times}_{p}$ le caract\`ere lisse non ramifi\'e qui envoie $p$ sur $\lambda$, et pour tout entier $n \in \ZZ$, on d\'esigne par $\chi^{s}_{n}$ le $\FFF_{p}$-caract\`ere du groupe fini $B(\FF_{p})$ d\'efini par $\chi^{s}_{n}\left(\left(\begin{array}{cc} a & b \\ 0 & d\end{array}\right)\right) := d^{n}$. On note enfin $\frak{a}$ le $\FFF_{p}$-caract\`ere de $B(\FF_{p})$ d\'efini par $\frak{a}\left(\left(\begin{array}{cc} a & b \\ 0 & d\end{array}\right)\right) := ad^{-1}$.
\begin{remark}
Les caract\`eres $\chi^{s}_{n}$ et $\frak{a}$ d\'efinissent par restriction des caract\`eres du groupe fini $B_{S}(\FF_{p}) = B(\FF_{p}) \cap SL_{2}(\FF_{p})$, qui d\'efinissent par inflation des $\FFF_{p}$-caract\`eres lisses de $I_{S}$ que l'on note encore $\chi^{s}_{n}$ et $\frak{a}$. Notons qu'ils satisfont d'une part aux relations suivantes :
\begin{equation*}
\forall \ M \in B_{S}(\FF_{p}), \ \forall \ n \in \ZZ, \ \chi^{s}_{n}(M) = \chi^{s}_{-n}(M^{-1}) \ ,
\end{equation*}
et que l'on a d'autre part $\chi_{2}^{s}\frak{a} = \det$.
\end{remark}

Pour tout sous-groupe ouvert $H$ de $G$, on dispose d'un foncteur d'induction compacte $\mathrm{ind}_{H}^{G}$ qui co$\ddot{\text{\i}}$ncide avec le foncteur d'induction lisse $\ind{H}{G}$ lorsque $H$ est d'indice fini dans $G$. Pour toute $\FFF_{p}$-repr\'esentation lisse $\sigma$ de $H$ et tout vecteur $v \in \sigma$, on note $[1,v]$ l'\'el\'ement de $\mathrm{ind}_{H}^{G}(\sigma)$ de support \'egal \`a $H$ et qui envoie $I_{2}$ sur $v$.\\
Ceci reste vrai si l'on remplace $G$ par $K$ et que l'on consid\`ere le sous-groupe d'indice fini $H = \kop{}$ de $K$. Dans ce cas, si $\sigma$ est une $\FFF_{p}$-repr\'esentation lisse de $\kop{}$ et si $v$ est un \'el\'ement de $\sigma$, on pose alors, pour tout $\ell \in \{0, \ldots , p-1\}$,
\begin{equation}
\label{Fellv}
F_{\ell}(v) := \underset{\lambda_0\in\FF_{p}}{\sum}\lambda_0^{\ell}\maq{[\lambda_0]}{1}{1}{0}\stand{1}{v}\in\ind{\kop{}}{K}{\sigma} \ .
\end{equation}
Rappelons par ailleurs que toute repr\'esentation lisse irr\'eductible de $K_{S}$ sur $\FFF_{p}$ provient par inflation d'une repr\'esentation irr\'eductible $\sym{r}$ de $SL_{2}(\FF_{p})$, et que toute repr\'esentation lisse irr\'eductible de $K$ sur $\FFF_{p}$ provient par inflation d'une repr\'esentation irr\'eductible $\sym{r} \otimes \det^{m}$ de $GL_{2}(\FF_{p})$, avec $r \in \{0, \ldots, p-1\}$ et $m \in \{0, \ldots , p-2\}$ uniques. Par abus de notation, nous noterons encore $\sym{r}$ (resp. $\sym{r} \otimes \det$) la repr\'esentation de $K_{S}$ (resp. de $K$) correspondante.

\noindent Enfin, si $\tau$ est une repr\'esentation lisse de $K_{S}$ sur $\FFF_{p}$, on note $\big\{\text{soc}^i(\tau)\big\}_{i\in\NN}$ sa filtration par le socle, que l'on d\'efinit comme suit : $\text{soc}^0(\tau)$ est le $K_{S}$-socle de $\tau$, i.e. le sous-$\FFF_{p}[K_{S}]$-module semi-simple maximal contenu dans $\tau$. En supposant avoir construit $\text{soc}^i(\tau)$, on d\'efinit $\text{soc}^{i+1}(\tau)$ comme l'image r\'eciproque de $\text{soc}^0\big( \tau/\big(\text{soc}^i(\tau)\big) \big)$ par la projection canonique $\tau\twoheadrightarrow \tau/\big(\text{soc}^i(\tau)\big)$.
On utilisera alors la notation
\begin{eqnarray*}
\mathrm{soc}^1(\tau)\textbf{---}\mathrm{soc}^1(\tau)/\mathrm{soc}^0(\tau)\textbf{---}\dots\textbf{---}
\mathrm{soc}^{n+1}(\tau)/\mathrm{soc}^n(\tau)
\textbf{---}\dots
\end{eqnarray*}
pour d\'esigner la suite des facteurs gradu\'es successifs de la filtration par le socle de $\tau$.

\noindent Plus g\'en\'eralement, si $\tau$ est une $\FFF_{p}$-repr\'esentation lisse de $K_{S}$ munie d'une filtation croissante $\big\{\tau\big\}_{i\in\N}$, nous \'ecrirons
\begin{eqnarray*}
\text{SocFil}(\tau_0)\textbf{---}\text{SocFil}(\tau_1/\tau_0)\textbf{---}\dots\textbf{---}
\text{SocFil}(\tau_{i+1}/\tau_i)
\textbf{---}\dots
\end{eqnarray*}
pour signifier que :
\begin{itemize}
 \item [$i)$] la filtration par le socle de $\tau$ est induite par raffinemment par la filtration par le socle de chaque facteur gradu\'e $\tau_{i+1}/\tau_i$ ;
 \item[$ii)$] la suite des facteurs gradu\'es de la filtration par le socle de $\tau$ s'obtient en juxtaposant les suites de facteurs gradu\'es associ\'es aux filtrations par le socle de chaque $\tau_{i+1}/\tau_i$.
\end{itemize}

\subsection{Repr\'esentations modulo $p$ de $SL_{2}(\QQ_{p})$}
\label{RappelsRamla}
Les travaux pr\'esent\'es dans \cite{ramla} fournissent sans hypoth\`ese sur $p$ une classification exhaustive des classes d'isomorphisme des repr\'esentations lisses irr\'eductibles de $SL_{2}(\QQ_{p})$ sur $\FFF_{p}$, ainsi qu'une description de la restriction \`a $SL_{2}(\QQ_{p})$ des repr\'esentations lisses irr\'eductibles de $GL_{2}(\QQ_{p})$ sur $\FFF_{p}$. Ils permettent notamment d'obtenir l'\'enonc\'e suivant \cite[Th\'eor\`emes 0.1, 2.16 et Proposition 4.5]{ramla}.
\begin{theorem}
\label{ramla1}
Soit $\lambda \in \FFF^{\times}_{p}$ et soit $r \in \{0, \ldots , p-1\}$.
\begin{enumerate}
\item La restriction \`a $G_{S}$ d\'efinit un isomorphisme de $\FFF_{p}[G_{S}]$-modules :
$$ \mathrm{Res}_{G_{S}}^{G}(\ind{B}{G}(\mu_{\lambda} \otimes \mu_{\lambda^{-1}} \omega^{r})) \simeq \ind{B_{S}}{G_{S}}(\mu_{\lambda^{2}}\omega^{p-1-r}) \ .$$
\item Le $\FFF_{p}[G_{S}]$-module $\text{Ind}_{B_{S}}^{G_{S}}(\mu_{\lambda^{2}}\omega^{p-1-r})$ est r\'eductible si et seulement si le caract\`ere $\mu_{\lambda^{2}}\omega^{p-1-r}$ est trivial. L'induite parabolique $\ind{B_{S}}{G_{S}}(\1)$ est un $\FFF_{p}[G_{S}]$-module ind\'ecomposable de longueur $2$ admettant le caract\`ere trivial comme sous-objet et la repr\'esentation de Steinberg $St_{S}$ comme quotient.
\item La restriction \`a $G_{S}$ \'etablit un isomorphisme de $\FFF_{p}[G_{S}]$-modules de la repr\'esentation de Steinberg $St$ de $G$ vers la repr\'esentation de Steinberg $St_{S}$.
\item La restriction \`a $G_{S}$ induit un isomorphisme de $\FFF_{p}[G_{S}]$-modules :
\begin{equation}
\label{scindageramladessg}
\mathrm{Res}_{G_{S}}^{G}(\pi(r,0,1)) \simeq \pi_{r} \oplus \pi_{p-1-r} \ .
\end{equation}
\end{enumerate}
\end{theorem}
\noindent Dans cet \'enonc\'e, $\pi(r,0,1) := \textrm{Coker}(T_{r} : \mathrm{ind}_{KZ}^{G}(\sym{r}) \to \mathrm{ind}_{KZ}^{G}(\sym{r}))$ est la repr\'esentation supersinguli\`ere de $GL_{2}(\QQ_{p})$ de param\`etre $r$ \cite[Section 2.7]{breuil}. On d\'esigne par $\bar{f}$ l'image dans $\pi(r,0,1)$ d'un \'el\'ement $f \in \mathrm{ind}_{KZ}^{G}(\sym{r})$. Les repr\'esentations $\pi_{0}, \ldots , \pi_{p-1}$ forment quant \`a elles un syst\`eme explicite de repr\'esentants des classes d'isomorphisme des repr\'esentations supersinguli\`eres de $G_{S}$ \cite[Section 4.1]{ramla} qui v\'erifient en particulier les assertions suivantes \cite[Propositions 4.7 et 4.11]{ramla}.
\begin{theorem}
\label{ramla2}
Soit $r \in \{0, \ldots , p-1\}$. 
\begin{enumerate}
\item L'espace $\pi^{I_{S}(1)}_{r}$ des vecteurs $I_{S}(1)$-invariants de la repr\'esentation supersinguli\`ere $\pi_{r}$ est de dimension $1$ sur $\FFF_{p}$ et engendr\'e par $v_{r} := \overline{[I_{2}, x^{r}]}$.
\item Le sous-groupe d'Iwahori standard $I_{S}$ agit sur $\pi^{I_{S}(1)}_{r}$ par le caract\`ere $\omega^{r}$.
\end{enumerate}
\end{theorem}

\subsection{Structure interne des repr\'esentations modulo $p$ de $GL_{2}(\QQ_{p})$}
\label{RappelsStefano} 
On suppose d\'esormais $p \geq 5$. Nous rappelons maintenant quelques r\'esultats de \cite{stefano} concernant la structure interne des repr\'esentations lisses irr\'eductibles de $GL_{2}(\QQ_{p})$ sur $\FFF_{p}$, qui d\'ecrivent notamment la filtration par le $K$-socle des $\FFF_{p}$-repr\'esentations lisses irr\'eductibles de dimension infinie de $GL_{2}(\QQ_{p})$ \cite[Theorems 1.1 and 1.2]{stefano}.
\begin{theorem}
\label{stefano}
Soit $(r, \lambda) \in \{0, \ldots , p-1\} \times \FFF^{\times}_{p}$ une paire de param\`etres.
\begin{enumerate}
\item La repr\'esentation $\ind{B}{G}(\mu_{\lambda} \otimes \mu_{\lambda^{-1}}\omega^{r})$ admet la filtration par le $K$-socle suivante :
	\begin{eqnarray*}
	\mathrm{SocFil}(\ind{I}{K}(\chi_{r}^{s}))\textbf{---}\mathrm{SocFil}(\ind{I}{K}(\chi_{r}^{s}\mathfrak{a}))\textbf{---}
	\mathrm{SocFil}(\ind{I}{K}(\chi_{r}^{s}\mathfrak{a}^2))\textbf{---}\dots \ .
	\end{eqnarray*}

\item La repr\'esentation de Steinberg $St$ de $GL_{2}(\QQ_{p})$ admet la filtration par le $K$-socle suivante : 
	\begin{eqnarray*}
	\sym{p-1}\textbf{---}\mathrm{SocFil}(\ind{I}{K}(\frak{a}))\textbf{---}\mathrm{SocFil}(\ind{I}{K}(\frak{a}^{2}))\textbf{---} 
	\mathrm{SocFil}(\ind{I}{K}(\frak{a}^{3}))\textbf{---} \dots \ .
	\end{eqnarray*}
\item La restriction \`a $KZ$ de la repr\'esentation supersinguli\`ere $\pi(r,0,1)$ est scind\'ee de longueur $2$ : 
\begin{equation}
\label{scindagestefanodessg}
\pi(r,0,1)\vert_{KZ} \simeq \mathrm{\Pi}_{r} \oplus \mathrm{\Pi}_{p-1-r} \ .
\end{equation}
Les repr\'esentations $\mathrm{\Pi}_{r}$ et $\mathrm{\Pi}_{p-1-r}$ admettent respectivement les filtrations par le $K$-socle suivantes :
	\begin{eqnarray*}
	\sym{r}\textbf{---}\mathrm{SocFil}(\ind{I}{K}(\chi_{r}^{s}\mathfrak{a}^{r+1}))\textbf{---}
	\mathrm{SocFil}(\ind{I}{K}(\chi_{r}^{s}\mathfrak{a}^{r+2}))
	\textbf{---}\mathrm{SocFil}(\ind{I}{K}(\chi_{r}^{s}\mathfrak{a}^{r+3}))\textbf{---}\dots
	\end{eqnarray*}
	et
	\begin{eqnarray*}
	\sym{p-1-r}\textbf{---}\mathrm{SocFil}(\ind{I}{K}(\chi_{r}^{s}\mathfrak{a}))\textbf{---}
	\mathrm{SocFil}(\ind{I}{K}(\chi_{r}^{s}\mathfrak{a}^{2}))
	\textbf{---}\mathrm{SocFil}(\ind{I}{K}(\chi_{r}^{s}\mathfrak{a}^{3}))\textbf{---}\dots \ .
	\end{eqnarray*}
\end{enumerate}
\end{theorem}

\section{De $K$ \`a $K_{S}$}
\label{main}
Cette section a pour but de combiner les r\'esultats rappel\'es ci-avant afin d'obtenir deux \'enonc\'es techniques nous permettant de d\'emontrer facilement le Th\'eor\`eme \ref{thmppaldelanote}. Nous proc\'edons comme suit : nous commen\c{c}ons par \'etudier plus en d\'etail la nature des $K$-extensions entre poids de Serre contenus dans une repr\'esentation lisse irr\'eductible donn\'ee de $G$, ce qui fournit un raffinement de l'\'enonc\'e du Th\'eor\`eme \ref{stefano}. Nous prouvons ensuite que ces $K$-extensions restent non scind\'ees lorsqu'on se limite \`a les consid\'erer comme des $K_{S}$-extensions.

\subsection{Structure des $K$-extensions entre poids de Serre}
\label{main1}
Traitons tout d'abord le cas d'une repr\'esentation de la forme $\ind{B}{G}(\mu_{\lambda} \otimes \mu_{\lambda^{-1}}\omega^{r})$ avec $\lambda \in \FFF_{p}^{\times}$ et $r \in \{0, \ldots , p-1\}$ fix\'es.  Un calcul direct reposant sur la d\'ecomposition d'Iwahori $G = KB = BK$ montre, par application de la d\'ecomposition de Mackey, que le $\FFF_{p}[K]$-module port\'e par $\ind{B}{G}(\mu_{\lambda} \otimes \mu_{\lambda^{-1}}\omega^{r})$ est isomorphe \`a $\ind{K\cap B}{K}{\chi^s_r}$. D'apr\`es \cite[\S 10]{stefano}, on dispose en outre d'un isomorphisme $K$-\'equivariant\footnote{Donc en particulier $I = K_{0}(p)$-\'equivariant.}
\begin{equation*}
\underset{\underset{n\in \N}{\longrightarrow}}{\lim}\,\ind{\kop{n+1}}{K}{\chi^s_r}\stackrel{\sim}{\longrightarrow}\ind{K\cap B}{K}{\chi^s_r} \ .
\end{equation*}
Par continuit\'e du foncteur d'induction $\ind{\kop{}}{K}$, on en d\'eduit l'existence d'un isomorphisme de $\FFF_{p}[K]$-modules
\begin{equation}
\label{isom1main1}
\ind{\kop{}}{K}{\bigg(\underset{\underset{n\in \N}{\longrightarrow}}{\lim}\,\ind{\kop{n+1}}{\kop{}}{\chi^s_r}\bigg)}\stackrel{\sim}{\longrightarrow}\ind{K\cap B}{K}{\chi^s_r} \ .
\end{equation}
Par ailleurs, on sait gr\^ace \`a \cite[Proposition 5.10]{stefano} que pour tout entier $n \geq 1$, la repr\'esentation $\ind{\kop{n+1}}{\kop{}}{\chi^s_r}$ est unis\'erielle\footnote{Ce qui signifie que sa restriction au sous-groupe $U(\ZZ_{p})$ des matrices unipotentes de $B \cap K$ est une suite d'extensions non scind\'ees d'objets irr\'eductibles.} de dimension $p^n$ sur $\FFF_{p}$ et qu'elle admet la filtration par le $\kop{}$-socle suivante :
\begin{equation*}
\chi_{r}^{s}\textbf{---}\chi_{r}^{s}\mathfrak{a}\textbf{---}\chi_{r}^{s}\mathfrak{a}^2\textbf{---}...\textbf{---}\chi^s_r\mathfrak{a}^{p^n-2}\textbf{---}\chi^s_r\mathfrak{a}^{p^n-1} \ .
\end{equation*}
On obtient donc ainsi le premier r\'esultat suivant.
\begin{lemma}
\label{lemmeind}
Soit $(r, \lambda) \in \{0, \ldots , p-1\} \times \FFF^{\times}_{p}$ une paire de param\`etres. Le $\FFF_{p}[K]$-module port\'e par $\ind{B}{G}(\mu_{\lambda} \otimes \mu_{\lambda^{-1}}\omega^{r})$ est isomorphe \`a $\ind{\kop{}}{K}{R_{\infty}^{-}}$, o\`u $R_{\infty}^{-} := \underset{\underset{n\in \N}{\longrightarrow}}{\lim}\,\ind{\kop{n+1}}{\kop{}}{\chi^s_r}$ est une repr\'esentation unis\'erielle de $\kop{}$ de longueur infinie admettant la filtration par le $\kop{}$-socle suivante :
$$
\chi_{r}^{s}\textbf{---}\chi_{r}^{s}\mathfrak{a}\textbf{---}\chi_{r}^{s}\mathfrak{a}^2\textbf{---}... \ . 
$$
\end{lemma}

\noindent La situation est analogue pour les $\FFF_{p}[KZ]$-modules $\mathrm{\Pi}_r$ et $\mathrm{\Pi}_{p-1-r}$ introduits dans l'\'enonc\'e du Th\'eor\`eme \ref{stefano}. Pour tout entier $n \in \NN$, notons $\sigma_{r}^{n}$ le $\FFF_{p}[\kop{n}]$-module d'espace sous-jacent $\displaystyle \bigoplus_{i = 0}^{r} \FFF_{p}X^{i}Y^{r-i}$ sur lequel l'action de $\kop{n}$ est donn\'ee par la formule suivante \cite[\S 3.1]{stefano} : 
\begin{equation*}
\sigma_{r}^{n}\left(\left(\begin{array}{cc} a & b \\ p^{n}c & d \end{array}\right)\right) \cdot X^{i}Y^{r-i} := (dX + cY)^{i}(p^{n}bX + aY)^{r-i} \ . 
\end{equation*}
En posant $R_{n} := \ind{\kop{n}}{K}(\sigma_{r}^{n})$, on dispose gr\^ace \`a \cite[Lemma 3.4]{stefano1} d'une injection de syst\`emes inductifs
\begin{eqnarray*}
\xymatrix{
\ar@{..>}[d]&\ar@{..>}[d]\\
R_0^-\oplus_{R^-_1}\dots\oplus_{R_{n-2}^-}R_{n-1}^-\ar@{^{(}->}[r]\ar@{^{(}->}[d] & R_0\oplus_{R_1}\dots\oplus_{R_{n-2}}R_{n-1}\ar@{^{(}->}[d] \\
R_0^-\oplus_{R^-_1}\dots\oplus_{R_{n}^-}R_{n+1}^-\ar@{^{(}->}[r]\ar@{..>}[d] & R_0\oplus_{R_1}\dots\oplus_{R_{n}}R_{n+1}\ar@{..>}[d]\\
&
}
\end{eqnarray*}
qui induit un isomorphisme $K$-\'equivariant \cite[Proposition 3.9]{stefano}
\begin{eqnarray*}
\underset{\underset{n \text{ impair}}{\longrightarrow}}{\lim}R_0\oplus_{R_1}\dots\oplus_{R_{n}}R_{n+1}\stackrel{\sim}{\longrightarrow} \mathrm{\Pi}_r \ .
\end{eqnarray*}
Chaque $R_0^-\oplus_{R^-_1}\dots\oplus_{R_{n}^-}R_{n+1}^-$ d\'efinit alors une repr\'esentation unis\'erielle de dimension finie sur $\FFF_{p}$ dont la filtration par le $\kop{}$-socle est de la forme \cite[Theorem 5.18]{stefano1} :
$$
\chi_{r}^{s}\mathfrak{a}^r\textbf{---}\chi_{r}^{s}\mathfrak{a}^{r+1}\textbf{---}\chi_{r}^{s}\mathfrak{a}^{r+2}\textbf{---}... \ .
$$
Le lemme suivant \'etablit une relation importante entre $R_0^-\oplus_{R^-_1}\dots\oplus_{R_{n}^-}R_{n+1}^-$ et $R_0\oplus_{R_1}\dots\oplus_{R_{n}}R_{n+1}$.
\begin{lemma}
\label{prepa}
Pour tout entier impair $n \geq -1$, on dispose de la suite exacte courte suivante de $\FFF_{p}[K]$-modules :
\begin{equation}
\label{sesprepa}
0\rightarrow \sym{p-1-r}\otimes{\det}^r\rightarrow \ind{\kop{}}{K}{(R_0^-\oplus_{R^-_1}\dots\oplus_{R_{n}^-}R_{n+1}^-)}\rightarrow
R_0\oplus_{R_1}\dots\oplus_{R_{n}}R_{n+1}\rightarrow 0.
\end{equation}
\end{lemma}
\begin{proof} La d\'emonstration s'effectue par r\'ecurrence sur l'entier impair $n \geq -1$, le cas $n = -1$ s'obtenant par un calcul direct qui repose sur la d\'efinition de $R_{0} \simeq \sym{r}$ et de $R_{0}^{-} = \chi^{s}_{r}\frak{a}^{r}$ \cite[\S 3]{stefano1}.\\
Le cas g\'en\'eral se traite \`a l'aide du diagramme commutatif suivant \cite[Proposition 3.5]{stefano1} : 
\begin{eqnarray*}
\def\objectstyle{\scriptscriptstyle}
\xymatrix @C=.05pc {
&&0\ar[rr]&&R_n^{-}\ar[rrr]\ar[ddl]\ar@{->>}'[dd][ddd]&&&R_{n+1}^{-}\ar[rrr]\ar[ddl]\ar@{->>}'[dd][ddd]&&&
R_{n+1}^{-}/R_{n}^{-}\ar[rr]\ar@{=}'[dd][ddd]\ar[ddl]&&0
\\
\\
&0\ar[rr]&&R_n\ar@{->>}[ddd]\ar[rrr]&&&R_{n+1}\ar@{->>}[ddd]\ar[rrr]&&&
R_{n+1}/R_n\ar@{=}[ddd]\ar[rr]&&0&\\
&&0\ar'[r][rr]&&
\dots\oplus_{R_{n-2}^{-}}R_{n-1}^{-}
\ar'[rr][rrr]\ar@{^{(}->}[ddl]
&&&
\dots\oplus_{R_{n}^{-}}R_{n+1}^{-}
\ar'[rr][rrr]\ar@{^{(}->}[ddl]&&&
R_{n+1}^{-}/R_n^{-}\ar[rr]\ar@{^{(}->}[ddl]&&0
\\
\\
&0\ar[rr]&&
(\dots\oplus_{R_{n-2}}R_{n-1})\vert_{\kop{}}\ar[rrr]&&&
(\dots\oplus_{R_{n2}}R_{n+1})\vert_{\kop{}}\ar[rrr]&&&
R_{n+1}/R_{n}\ar[rr]&&0 \ .&
}
\end{eqnarray*}
Par transitivit\'e de l'induction compacte et \cite[\S 3]{stefano1}, on sait qu'il existe un isomorphisme de $\FFF_{p}[K]$-modules $\ind{\kop{}}{K}{R_{n+1}^-}\cong R_{n+1}$. On d\'eduit donc de l'exactitude du foncteur $\ind{K_{0}(p)}{K}$ et de la r\'eciprocit\'e de Frobenius compacte l'existence du diagramme commutatif \`a lignes exactes suivant :
\begin{eqnarray*}
\xymatrix{
0\ar[r]&\dots\oplus_{R_{n-1}}R_{n-1}\ar[r]
&\dots\oplus_{R_{n}}R_{n+1}
\ar[r]&
R_{n+1}/R_{n}\ar[r]&0\\
0\ar[r]&\ind{\kop{}}{K}{
(\dots\oplus_{R_{n-1}^-}R_{n-1}^-)}
\ar[r]\ar@{->>}[u]&\ind{\kop{}}{K}{(
\dots\oplus_{R_{n+1}^-}R_{n+1}^-)}\ar[r]\ar@{->>}[u]&
R_{n+1}/R_{n}\ar[r]\ar@{=}[u]&0 \ .
}
\end{eqnarray*}
Ainsi, si la suite exacte courte \eqref{sesprepa} est v\'erifi\'ee pour l'entier $n-1 \geq -1$, l'application du lemme du serpent au diagramme ci-dessus implique directement que la suite exacte \eqref{sesprepa} est v\'erifi\'ee pour l'entier $n+1$, ce qui ach\`eve la d\'emonstration. \qed
\end{proof}

\noindent La continuit\'e et l'exactitude du foncteur $\ind{K_{0}(p)}{K}$ nous permettent en particulier d'en d\'eduire le r\'esultat suivant.
\begin{corollary}
\label{proposition3.3} Soit $r \in \{0, \ldots p-1\}$ et soit $\mathrm{\Pi}_{r}$ le facteur direct du $\FFF_{p}[KZ]$-module $\pi(r,0,1)$ de $K$-socle \'egal \`a $\sym{r}$. On dispose alors d'une suite exacte courte de $\FFF_{p}[K]$-modules
\begin{equation*}
0\longrightarrow \sym{p-1-r}\otimes{\det}^r\longrightarrow \ind{\kop{}}{K}{(R_{\infty,r}^-)}\longrightarrow \mathrm{\Pi}_r\longrightarrow0 \ , 
\end{equation*}
o\`u $R_{\infty,r}^- := \underset{\underset{n \text{ impair}}{\longrightarrow}}{\lim}R^-_0\oplus_{R^-_1}\dots\oplus_{R^-_{n}}R^-_{n+1}$ est une repr\'esentation unis\'erielle de dimension infinie sur $\FFF_{p}$ dont la filtration par le $\kop{}$-socle est donn\'ee par
\begin{equation*}
\chi_{r}^{s}\mathfrak{a}^r\textbf{---}\chi_{r}^{s}\mathfrak{a}^{r+1}\textbf{---}\chi_{r}^{s}\mathfrak{a}^{r+2}\textbf{---}... \ .
\end{equation*}
\end{corollary}

\subsection{Restriction \`a $K_{S}$}
\label{main2}
Nous allons maintenant \'etudier la restriction \`a $K_{S}$ des extensions entre poids de Serre consid\'er\'ees dans la sous-section pr\'ec\'edente. Commen\c{c}ons par rappeler un r\'esultat d\^u \`a Pa$\check{\text{s}}$k$\bar{\text{u}}$nas \cite[Proposition 5.4]{Pask}, qui traite le cas des $K_{0}(p)$-extensions entre caract\`eres.
\begin{proposition}
\label{pasku}
Soient $\chi$ et $\psi$ deux $\FFF_{p}$-caract\`eres lisses de $\kop{}$. L'espace d'extensions $\mathrm{Ext}^1_{\kop{}/Z_1}(\psi,\chi)$ est non nul si et seulement si $\psi=\chi\mathfrak{a}$. Dans ce cas, c'est un espace de dimension $1$ sur $\FFF_{p}$, et toute extension non triviale de $\chi\mathfrak{a}$ par $\chi$ admet une base $(e_0,e_1)$ sur $\FFF_{p}$ pour laquelle $K_{0}(p)$ agit sur $e_i$ par le
caract\`ere $\chi\mathfrak{a}^{i}$ et telle que
\begin{eqnarray*}
\maq{1+pa}{b}{pc}{1+pd}\cdot e_1=e_1+ce_0 \ .
\end{eqnarray*}
\end{proposition}
\noindent Fixons un param\`etre $r\in\{0,\dots,p-1\}$. La structure de $\FFF_{p}[K]$-module port\'ee par $\ind{\kop{}}{K}{\chi_r^s}$ est donn\'ee par \cite[Lemma 2.3 \& Theorem 2.4]{BP} et admet la description suivante : 
\begin{eqnarray}
\sym{r}\textbf{---}\sym{p-1-r}\otimes{\det}^r&\textrm{ si }&r\not\equiv 0 \text{ modulo}\ (p-1) \ ,\label{regular}\\
\sym{p-1}\oplus\sym{0}&\textrm{ si }&r\equiv 0 \text{ modulo}\ (p-1) \ . \label{nonregular}
\end{eqnarray}
Les fonctions $F_{\ell}(v)$ d\'efinies par la formule \eqref{Fellv} permettent en outre d'expliciter certains espaces de vecteurs invariants ou coinvariants sous $K_{0}(p)$. Plus pr\'ecis\'ement, en fixant un vecteur non nul $e \in \chi_{r}^{s}$, on dispose gr\^ace \`a \cite[Lemma 2.5]{BP} des formules suivantes lorsque $r\not\equiv 0 \text{ modulo}\, (p-1)$ :
\begin{eqnarray}
(\mathrm{soc}_K(\ind{\kop{}}{K}{\chi_r^s}))^{\kop{}}&=&\langle F_{0}(e)\rangle_{\FFF_{p}} \ ;
\label{primo}\\
(\mathrm{cosoc}_K(\ind{\kop{}}{K}{\chi_r^s}))_{\kop{}}&=&\langle F_{p-1}(e)\rangle_{\FFF_{p}} \ . 
\label{secondo}
\end{eqnarray}
L'isomorphisme de $\FFF_{p}[K]$-modules $\sym{p-1}\oplus\sym{0}\cong\ind{\kop{}}{K}{\1}$ permet de plus de d\'eduire de \cite[Lemma 2.6]{BP} les cas $r = 0$ et $r = p-1$ : 
\begin{eqnarray}
(\sym{p-1})^{\kop{}}&=&\langle F_{0}(e)\rangle_{\FFF_{p}}
\label{primononregular}\\
(\sym{p-1})_{\kop{}}&=&\langle F_{p-1}(e)+\stand{1}{e}\rangle_{\FFF_{p}}
\label{secondononregular}\\
\sym{0}&=&\langle F_{0}(e)+\stand{1}{e}\rangle_{\FFF_{p}}
\label{terzononregular}
\end{eqnarray}
Nous allons \`a pr\'esent utiliser la Proposition \ref{pasku} dans le cas $\chi = \chi_{r}^{s}$ pour construire nos extensions entre poids de Serre. Par exactitude du foncteur $\ind{K_{0}(p)}{K}$, l'extension non triviale $E$ de $\chi_{r}^{s}\mathfrak{a}$ par $\chi_{r}^{s}$ d\'ecrite dans la Proposition \ref{pasku} d\'efinit une $K$-extension de la forme 
\begin{equation*}
0 \longrightarrow \mathrm{Ind}_{\kop{}}^{K}(\chi_{r}^{s}) \longrightarrow \ast \stackrel{pr}{\longrightarrow} \mathrm{Ind}_{\kop{}}^{K}(\chi_{r}^{s}\frak{a}) \longrightarrow 0 \ . 
\end{equation*}
Pour $r \not=2$, on d\'efinit $E_{r}$ comme l'image r\'eciproque de $\mathrm{soc}_K(\ind{\kop{}}{K}(\chi_{r}^s\mathfrak{a}))$ dans l'extension induite
\begin{eqnarray}
\label{extension2}
 0 \longrightarrow  \mathrm{cosoc}_K(\ind{\kop{}}{K}(\chi_{r}^s)) \longrightarrow \ast \stackrel{pr}{\longrightarrow} \ind{\kop{}}{K}(\chi_{r}^s\frak{a}) \longrightarrow 0 \ ,
\end{eqnarray}
sous r\'eserve de poser $\mathrm{cosoc}_K(\ind{\kop{}}{K}(1)) \defeq \sym{p-1}$.\\
Pour $r=2$, auquel cas on a $\chi_{2}^{s}\frak{a} = \det$, on d\'efinit $E_{2,0}$ et $E_{2,p-1}$ comme les images r\'eciproques respectives de $\sym{0}\otimes {\det}$ et de $\sym{p-1}\otimes{\det}$ dans l'extension \eqref{extension2}.\\

\noindent En particulier, on voit que $E_r$ (pour $r\neq 2$), $E_{2,0}$ et $E_{2,p-1}$ poss\`edent respectivement les filtrations $K$-\'equivariantes suivantes\footnote{Les extensions sont ici not\'ees en pointill\'e car \`a ce stade, nous ne savons pas encore que ces extensions sont non scind\'ees : c'est l'objet du Corollaire \ref{nonscindage}. } : 
\begin{eqnarray*}
(E_r) & :&\sym{\lfloor p-1- r\rfloor}\otimes{\det}^r\cdots\sym{\lfloor r-2\rfloor}\otimes{\det} \ ;\\
(E_{2,0}) & :&\sym{p-3}\otimes{\det}^2\cdots{\det} \ ; \\
(E_{2,p-1}) & :&\sym{p-3}\otimes{\det}^2\cdots{\sym{p-1}\otimes{\det}} \ .
\end{eqnarray*}
Rappelons maintenant que pour tous scalaires $\lambda_0,\mu\in\fp$, on a l'\'egalit\'e suivante dans $\ZZ_{p}/p^2\ZZ_{p}$ \cite[Chap. IX, \S 1, n$^{\circ}$3]{Bourb} :
\begin{eqnarray}
\label{witt}
[\lambda_0]+[\mu]\equiv [\lambda_0+\mu]+p[P(\lambda_0,\mu)]\quad\text{mod}\, p^2 
\end{eqnarray}
avec $P(\lambda_0,\mu)\in\fp[\lambda_0,\mu]$ polyn\^ome de degr\'e $p-1$ en $\lambda_{0}$ et de coefficient dominant \'egal \`a $\pm\mu$.\\

\noindent Nous pouvons alors d\'emontrer le r\'esultat technique suivant, sur lequel repose notre preuve du Th\'eor\`eme \ref{thmppaldelanote}.
\begin{lemma}
\label{lemcrucial}
Les repr\'esentations de $U(\ZZ_{p})$ port\'ees par $E_{r}$, $E_{2,0}$ et $E_{2,p-1}$ sont unis\'erielles.
\end{lemma}
\begin{proof}
Notons $\overline{E}_r$, $\overline{E}_{2,0}$ et $\overline{E}_{2,p-1}$ les quotients respectifs des $\FFF_{p}[U(\ZZ_{p})]$-modules $E_r$, $E_{2,0}$ et $E_{2,p-1}$ par
le noyau de la projection naturelle sur les $U(\zp)$-coinvariants 
\begin{equation*}
\sym{\lfloor p-1- r\rfloor}\otimes{\det}^r\twoheadrightarrow(\sym{p-1-\lfloor r\rfloor}\otimes{\det}^r)_{U(\zp)} \ .
\end{equation*}
Ce sont des repr\'esentations de $U(\zp)$ qui satisfont par construction aux suites exactes courtes suivantes de $\FFF_{p}[U(\zp)]$-modules :
\begin{eqnarray}
0\longrightarrow (\sym{\lfloor p-1- r\rfloor}\otimes{\det}^r)_{U(\zp)}\longrightarrow\overline{E}_r\longrightarrow\sym{\lfloor r-2\rfloor}\longrightarrow 0 \ ; & \  \label{EbarR}\\
0\longrightarrow (\sym{p-3}\otimes{\det}^2)_{U(\zp)}\longrightarrow\overline{E}_{2,0}\longrightarrow \1\longrightarrow 0 \ ; & \ \label{Ebar20}\\
0\longrightarrow (\sym{p-3}\otimes{\det}^2)_{U(\zp)}\longrightarrow\overline{E}_{2,p-1}\longrightarrow \sym{p-1} \longrightarrow 0 \ . & \ \label{Ebar2p-1}  
\end{eqnarray}
On peut donc d\'efinir $\overline{E}_{r}'$ (resp. $\overline{E}_{2,0}'$, $\overline{E}_{2,p-1}'$) comme la sous-repr\'esentation de $\overline{E}_r$ (resp. $\overline{E}_{2,0}$, $\overline{E}_{2,p-1}$) obtenue par image reciproque de $(\sym{\lfloor r-2\rfloor})^{U(\zp)}$ (resp. $\1$, $(\sym{p-1})^{U(\zp)}$) \`a partir de la suite exacte \eqref{EbarR} (resp. \eqref{Ebar20}, \eqref{Ebar2p-1}). La restriction \`a $U(\zp)$ des poids de Serre de $GL_{2}(\ZZ_{p})$ \'etant unis\'erielle, il nous suffit de prouver que les $U(\zp)$-extensions $\overline{E}_{r}'$, $\overline{E}_{2,0}'$ et $\overline{E}_{2,p-1}'$ ne sont pas scind\'ees pour conclure.\\
Pour ce faire, consid\'erons une base $(e_{0}, e_{1})$ adapt\'ee \`a l'extension non triviale $E$ au sens de la Proposition \ref{pasku}. Les identit\'es (\ref{primo}), (\ref{secondo}), (\ref{primononregular}), (\ref{secondononregular}) et (\ref{terzononregular}) permettent alors de voir que l'on a 
\begin{eqnarray*}
&&\overline{E}_{r}' = \langle F_{p-1}(e_{0}), F_{0}(e_{1})\rangle_{\fpbar} \ , \\
&&\overline{E}_{2,0}' = \langle F_{p-1}(e_{0}), F_{0}(e_{1})+\stand{1}{e}\rangle_{\fpbar} \ , \\
&&\overline{E}_{2,p-1}' = \langle F_{p-1}(e_{0}), F_{0}(e_{1})\rangle_{\fpbar} \ ,
\end{eqnarray*}
les fonctions $F_{p-1}(e_{0})$ et $F_{0}(e_{1})$ \'etant ici des \'el\'ements de $\ind{K_{0}(p)}{K}(E)$.\\

\noindent \'Etant donn\'e que $1$ est un g\'en\'erateur topologique de $\zp$, l'action de $U(\zp)$ sur $\overline{E}_{r}'$, $\overline{E}_{2,0}'$ et $\overline{E}_{2,p-1}'$ est enti\`erement d\'etermin\'ee par l'action des matrices de la forme $\maq{1}{[\mu]}{0}{1}$ avec $\mu\in\fp$. En remarquant que $F_{\ell}(e_{0})$ est nul dans $\overline{E}_{r}'$, $\overline{E}_{2,0}'$ et $\overline{E}_{2,p-1}'$ lorsque $\ell\neq p-1$, on obtient par un calcul direct reposant sur l'\'egalit\'e modulaire \eqref{witt} et sur la d\'efinition des vecteurs $(e_{0}, e_{1})$ que l'on a, pour tout scalaire non nul $\mu \in \FF_{p}$,
$$
\left(\maq{1}{[\mu]}{0}{1}-I_{2}\right)\cdot F_{0}(e_{1})=\pm\mu F_{p-1}(e_0) \not= 0 \ .
$$
Ceci ach\`eve la d\'emonstration une fois que l'on a remarqu\'e que $F_{p-1}(e_0)$ est fixe sous l'action de $U(\zp)$. \qed
\end{proof}
\begin{corollary}
\label{nonscindage}
\begin{enumerate}
\item Les $K_{S}$-extensions port\'ees par $E_{r}$ (pour $r\neq 2$), $E_{2,0}$ et $E_{2,p-1}$ ne sont pas scind\'ees.
\item La filtration par le $K$-socle de l'induite $\ind{\kop{}}{K}{\chi^s_r}$ est stable par restriction \`a $K_{S}$.
\end{enumerate}

\end{corollary}
\begin{proof}
La premi\`ere assertion est une cons\'equence directe du Lemme \ref{nonscindage} puisque $U(\ZZ_{p})$ est un sous-groupe de $K_{S}$. La seconde assertion d\'ecoule quant \`a elle du Lemme \ref{nonscindage} et de la d\'ecomposition de Mackey qui assure que l'on a 
$$ \ind{I}{K}(\chi_{r}^{s})\vert_{U(\zp)} \simeq \rho \oplus \1 $$
avec $\rho$ repr\'esentation unis\'erielle de $U(\ZZ_{p})$ de dimension $p$ sur $\FFF_{p}$. \qed
\end{proof}

\section{Preuve du Th\'eor\`eme \ref{thmppaldelanote}}
Nous pouvons maintenant d\'emontrer notre r\'esultat principal, dont nous rappelons l'\'enonc\'e ci-dessous.
\begin{theorem}
Soit $r \in \{0, \ldots , p-1\}$ et soit $\lambda \in \FFF^{\times}_{p}$.
\begin{enumerate}
\item La repr\'esentation $\ind{B_{S}}{G_{S}}(\mu_{\lambda}\omega^{p-1-r})$ admet la filtration par le $K_{S}$-socle suivante : 
	\begin{eqnarray*}
	\mathrm{SocFil}(\ind{I_S}{K_S}(\chi_{r}^{s}))\textbf{---}\mathrm{SocFil}(\ind{I_S}{K_S}(\chi_{r-2}^s))\textbf{---}
	\mathrm{SocFil}(\ind{I_S}{K_S}(\chi_{r-4}^s))\textbf{---}\dots.
	\end{eqnarray*}

\item La repr\'esentation de Steinberg $St_{S}$ admet la filtration par le $K_{S}$-socle suivante : 
	\begin{eqnarray*}
	\sym{p-1}\textbf{---}\mathrm{SocFil}(\ind{I_{S}}{K_{S}}(\frak{a}))\textbf{---}\mathrm{SocFil}(\ind{I_{S}}{K_{S}}(\frak{a}^{2}))\textbf{---} 
	\mathrm{SocFil}(\ind{I}{K}(\frak{a}^{3}))\textbf{---} \dots \ .
	\end{eqnarray*}

\item La repr\'esentation supersinguli\`ere $\pi_{r}$ admet la filtration par le $K_{S}$-socle suivante :
	\begin{eqnarray*}
	\sym{r}\textbf{---}\mathrm{SocFil}(\ind{I_S}{K_S}(\chi_{-r-2}^s))\textbf{---}
	\mathrm{SocFil}(\ind{I_S}{K_S}(\chi_{-r-4}^s))
	\textbf{---}\mathrm{SocFil}(\ind{I_S}{K_S}(\chi_{-r-6}^s))\textbf{---}\dots
	\end{eqnarray*}
\end{enumerate}
\end{theorem}

\subsection{Cas des repr\'esentations non supersinguli\`eres}
Fixons $r \in \{0, \ldots , p-1\}$ et $\lambda \in \FFF_{p}^{\times}$. Par exactitude du foncteur d'induction compacte, le Lemme \ref{lemmeind} assure que le $\FFF_{p}[K]$-module port\'e par $\ind{B}{G}(\mu_{\sqrt{\lambda}} \otimes \mu_{\sqrt{\lambda^{-1}}}\omega^{r})$ admet une filtration $K$-\'equivariante de longueur infinie dont les facteur gradu\'es sont d\'ecrit par 
\begin{equation*}
\ind{I}{K}(\chi_r^s)\text{---}\ind{I}{K}(\chi_r^s\mathfrak{a})\text{---}\ind{I}{K}(\chi_r^s\mathfrak{a}^2)\text{---}\dots \ .
\end{equation*}
Par suite, les extensions entre les poids de Serre qui apparaissent dans la filtration par le $K$-socle de $\ind{B}{G}(\mu_{\sqrt{\lambda}} \otimes \mu_{\sqrt{\lambda^{-1}}}\omega^{r})$ sont, \`a torsion par un $\FFF_{p}$-caract\`ere lisse pr\`es, des extensions non scind\'ees de la forme $E_{i}$, $E_{2,0}$, $E_{2,p-1}$ ou $\ind{\kop{}}{K}{\chi_{r}^{s}\mathfrak{a}^j}$ avec $i \in \{0, \ldots , p-1\}$ diff\'erent de $2$ et $j \in \NN$ v\'erifiant $r - 2j \not\equiv 0 \text{ modulo } (p-1)$. Le Lemme \ref{nonscindage} affirme alors que ces objets d\'efinissent des $K_{S}$-extensions non scind\'ees, ce qui prouve le r\'esultat voulu pour les repr\'esentations de la forme $\ind{B_{S}}{G_{S}}(\mu_{\lambda}\omega^{p-1-r})$ gr\^ace aux Th\'eor\`emes \ref{ramla1} et \ref{stefano} et \`a l'\'egalit\'e $I_{S} = \kop{} \cap K_{S}$.\\
On en d\'eduit directement le cas de la repr\'esentation de Steinberg en rappelant que l'on dispose de la suite exacte courte non scind\'ee de $\FFF_{p}[K]$-modules 
$$ 1 \longrightarrow \1 \longrightarrow \ind{B_{S}}{G_{S}}(\1) \longrightarrow St_{S} \longrightarrow 1 \ . $$
Comme le $K_{S}$-socle de la repr\'esentation $\ind{B_{S} \cap K_{S}}{K_{S}}(\1)$ est somme directe du caract\`ere trivial et du $K_{S}$-socle de la repr\'esentation obtenue par inflation de la repr\'esentation de Steinberg du groupe fini $SL_{2}(\FF_{p})$, celui-ci engendrant de plus le $\FFF_{p}[K_{S}]$-module irr\'eductible $St_{S}$, il suffit de remarquer qu'en restriction \`a $I_{S}$, on a $\chi_{2}^{s} = \frak{a}$ pour d\'eduire le r\'esultat voulu de la $K_{S}$-filtration de $\ind{B_{S}}{G_{S}}(\1)$\footnote{Qui correspond au cas $(r,\lambda) = (p-1,1)$.}.

\subsection{Cas des repr\'esentations supersinguli\`eres}
Pour terminer la preuve du Th\'eor\`eme \ref{thmppaldelanote}, nous allons relier les repr\'esentations supersinguli\`eres de $G_{S}$ aux repr\'esentations $\mathrm{\Pi}_{r}$ de $KZ$ d\'efinies dans le Th\'eor\`eme \ref{stefano}.
\begin{lemma}
\label{lemme2}
Soit $r \in \{0, \ldots , p-1\}$. Le $\FFF_{p}[KZ]$-module $\mathrm{\Pi}_{r}$ est stable sous l'action de $G_{S}$, et la repr\'esentation de $G_{S}$ qu'il d\'efinit est alors isomorphe \`a $\pi_{r}$.
\end{lemma}
\begin{proof} Un calcul direct bas\'e sur la d\'efinition de $\mathrm{\Pi}_{r}$ donn\'ee dans \cite[Proposition 3.9]{stefano} montre que $\mathrm{\Pi}_{r}$ est stable sous l'action de $ \alpha_{0} := \left(\begin{array}{lc} p^{-1} & 0 \\ 0 & p\end{array}\right)$. Comme la d\'ecomposition de Cartan assure que $G_{S}$ est engendr\'e par $K_{S}$ et par $\alpha_{0}$, cela suffit \`a prouver que $\mathrm{\Pi}_{r}$ est un sous-$\FFF_{p}[G_{S}]$-module de $\pi(r,0,1)$. Il suffit ensuite de remarquer que $v_{r}$ appartient \`a $\mathrm{\Pi}_{r}$ pour obtenir que $\pi_{r} = <G_{S} \cdot v_{r}>_{\FFF_{p}}$ est contenu dans $\mathrm{\Pi}_{r}$, puis de comparer les d\'ecompositions de $\pi(r,0,1)$ donn\'ees par les Th\'eor\`emes \ref{ramla1} et \ref{stefano} pour en conclure que les $\FFF_{p}[G_{S}]$-modules $\pi_{r}$ et $\mathrm{\Pi}_{r}$ sont isomorphes. \qed
\end{proof}

\noindent On d\'emontre alors le cas supersingulier du Th\'eor\`eme \ref{thmppaldelanote} comme suit : le Corollaire \ref{proposition3.3} et l'exactitude du foncteur d'induction compacte assurent que $\mathrm{\Pi}_{r}$ admet une filtration $K$-\'equivariante de longueur infinie dont les facteurs gradu\'e sont d\'ecrits par
\begin{equation*}
\sym{r}\text{---}\ind{I}{K}(\chi_r^s\mathfrak{a}^{r+1})\text{---}\ind{I}{K}(\chi_r^s\mathfrak{a}^{r+2})\text{---}\ind{I}{K}(\chi_r^s\mathfrak{a}^{r+3})\text{---}\dots \ .
\end{equation*}
Par suite, les extensions entre poids de Serre qui apparaissent dans la filtration par le $K$-socle de $\mathrm{\Pi}_{r}$ sont, \`a torsion par un caract\`ere lisse pr\`es, des extensions non scind\'ees de la forme $E_{i}$, $E_{2,0}$, $E_{2,p-1}$ ou $\ind{\kop{}}{K}{\chi_{r}^{s}\mathfrak{a}^j}$ avec $i \in \{0, \ldots , p-1\}$ diff\'erent de $2$ et $j \in \NN$ v\'erifiant $r - 2j \not\equiv 0 \text{ modulo } p-1$. Le Lemme \ref{nonscindage} affirme que ces objets d\'efinissent des $K_{S}$-extensions non scind\'ees, ce qui prouve le r\'esultat voulu gr\^ace au Lemme \ref{lemme2}, au Th\'eor\`eme \ref{stefano} et \`a l'\'egalit\'e $I_{S} = \kop{} \cap K_{S}$.

\subsection{Deux remarques}
\begin{enumerate}
\item[\tt $1)$] Les r\'esultats pr\'esent\'es dans cet article restent valables lorsque $p = 3$. La difficult\'e de ce cas r\'eside essentiellement dans la complexit\'e technique des calculs \`a mener pour obtenir la filtration par l'Iwahori-socle des repr\'esentations $R_{\infty}$ introduites dans la Section \ref{RappelsStefano}.
\item[\tt $2)$] Le groupe $SL_{2}(\QQ_{p})$ poss\`ede deux classes de conjugaison de sous-groupes ouverts compacts maximaux hypersp\'eciaux. Il est donc naturel de s'interroger quant \`a l'influence possible du choix de cette classe sur les r\'esultats obtenus ci-avant. Il suffit toutefois de remarquer que l'action de l'\'el\'ement $\alpha = \left(\begin{array}{cc} 1 & 0 \\ 0 & p \end{array}\right)$, dont l'action par conjugaison \'echange les deux classes de sous-groupes ouverts compacts maximaux sus-mentionn\'ees, \'echange aussi les composantes $\mathrm{\Pi}_{r}$ et $\mathrm{\Pi}_{p-1-r}$ introduites dans le Th\'eor\`eme \ref{stefano} pour v\'erifier imm\'ediatement que l'on dispose d'\'enonc\'es parfaitement analogues si l'on s'int\'eresse aux filtrations par le $K_{S}^{\alpha} := \alpha K_{S} \alpha^{-1}$-socle des repr\'esentations lisses irr\'eductibles de $SL_{2}(\QQ_{p})$ sur $\FFF_{p}$. Autrement dit, le choix de la classe de conjugaison de sous-groupes ouverts compacts maximaux hypersp\'eciaux consid\'er\'ee n'a pas d'influence sur les r\'esultats obtenus.
\item[\tt $3)$] Remarquons enfin que pour certaines valeurs de $r$, la filtration par le $K_{S}$-socle de la repr\'esentation supersinguli\`ere $\pi_{r}$ est donn\'ee par la suite suivante de poids de Serre:
\begin{eqnarray*}
	\sym{r}\textbf{---}\sym{p-3-r}\textbf{---}\sym{r+2}
	\textbf{---}\sym{p-5-r}\textbf{---}\sym{r+4}\textbf{---}\dots \ ,
\end{eqnarray*}
et semble ainsi, pour des raisons que nous ne savons pas encore expliquer, refl\'eter la structure de l'arbre de Brauer de $SL_{2}(\FF_{p})$ \cite[Section 9.9]{humph}.
\end{enumerate}

\end{document}